\documentclass[a4paper,12pt]{amsart}
\usepackage{amsmath,amssymb,latexsym,amsfonts,amscd}

\title{On 6-canonical map of irregular threefolds of general type}
\author{Jungkai Chen, Meng Chen, and Zhi Jiang}
\address{\rm National Center for Theoretical Sciences, Taipei Office, and
Department of Mathematics, National Taiwan University, Taipei,
106, Taiwan} \email{jkchen@math.ntu.edu.tw}

\address{\rm Institute of Mathematics \& LMNS, Fudan University, Shanghai 200433, People's Republic of
China} \email{mchen@fudan.edu.cn}

\address{Math\'ematiques B\^{a}timent 425\\
Universit\'{e} Paris-Sud\\
F-91405 Orsay, France} \email{zhi.jiang@math.u-psud.fr}

\thanks{The first author was partially supported by NCTS/TPE and
National Science Council of Taiwan. The second author was
supported by National Natural Science Foundation of China (\#11171068), Doctoral Fund of Ministry of Education of China (\#20110071110003)and partially by NSFC for Innovative Research Groups (\#11121101)}

\newcommand\Pic{\text{\rm Pic}}

\newcommand\ot{{\otimes}}
\newcommand\OO{{\mathcal{O}}}

\newcommand\cE{{\mathcal{E}}}
\newcommand\cF{{\mathcal{F}}}

\newcommand\cL{{\mathcal{L}}}
\newcommand\cM{{\mathcal{M}}}

\newcommand\cO{{\mathcal{O}}}

\newcommand\bP{{\mathbb P}}
\newcommand\bQ{{\mathbb Q}}

\newtheorem{thm}{Theorem}[section]
\newtheorem{lem}[thm]{Lemma}
\newtheorem{cor}[thm]{Corollary}

\theoremstyle{definition}
\newtheorem{defn}[thm]{Definition}

\newtheorem{exmp}[thm]{Example}

\theoremstyle{remark}

\begin{document}
\begin{abstract} We prove that, for any nonsingular projective irregular 3-fold of general type, the 6-canonical map is birational onto its image.
\end{abstract}
\maketitle
\pagestyle{myheadings} \markboth{\hfill J. Chen, M. Chen, and Z. Jiang
\hfill}{\hfill On 6K of irregular 3-folds\hfill}

\section{\bf Introduction}
Given a nonsingular projective variety $V$ of general type, by definition, the pluricanonical map $\varphi_m$ is birational for all sufficiently large integer $m$. It is natural and interesting to find an effective bound for $m$. By the result of Hacon-McKernan \cite{H-M}, Takayama \cite{Ta} and Tsuji (cf. \cite{Tsu}), one
knows that there exists a positive integer $r_n$ depending only on $n=\dim(V)$ such that $\varphi_m$ is birational for all $m \ge r_n$.
In the case of threefolds, the previous work of the first two authors (cf. \cite{CC1, CC2}) shows that $r_3\leq 73$.

In this note we  study irregular threefolds (i.e. $q(V)>0$) of general type. Recent developments on the technique inspired by the Fourier-Mukai
transform show that the geometry of irregular threefolds is very similar to that of general fibers of the Albanese map. Noting that the 5-canonical map of a general type surface is birational, one may expect that $\varphi_5$ is birational too for those threefolds which admit a fibration over  (a subvariety of) an abelian variety. Indeed,
given a nonsingular projective irregular  threefold  of general type, it has been proved by Chen and Hacon \cite[Theorem 2.8, Proposition 2.9]{JC-H-mz} that $\varphi_{m}$ is birational for all
$m\geq 7$ and, moreover,  that $\varphi_5$ is birational if $\chi(\omega_X)>0$.

The aim of this paper is to prove the following:
\begin{thm}\label{6K} Let $V$ be a nonsingular projective irregular 3-fold of general type. Then $\varphi_6$ is birational.
\end{thm}

\section{\bf  Proof of the main theorem}

\subsection{Reductions}\label{reduc} In order to prove Theorem \ref{6K}, we have the following reduction to special cases:
\begin{itemize}
\item[(1)] Let $V$ be a nonsingular projective 3-fold of general type. Take any birational projective model $W$ of $V$ so that $W$ has at worst canonical singularities. Then $V$ and $W$ share the same birational invariants and $\Phi_{mK_W}\approx\Phi_{mK_V}$. Therefore it is sufficient to prove the statement of Theorem \ref{6K} just replacing $V$ with any suitable birational model $W$.

\item[(2)] By Chen and Hacon \cite[Proposition 2.9]{JC-H-mz}, one only needs to consider the following situation (since, otherwise, $|6K|$ gives a birational map):
\begin{quote}
$(\natural)$ The Albanese map of $V$ induces the fibration $a_V
:V\longrightarrow C$ onto an elliptic curve $C$, of which the
general fiber is a $(1,2)$ surface $S$, i.e.
$(K_{S_0}^2,p_g(S))=(1,2)$, where $S_0$ is assumed to be the
minimal model of $S$.
\end{quote}

\item[(3)] Also due to Chen and Hacon \cite[Theorem 1.1]{JC-H-mz},
we may assume that $\chi(\OO_V)\geq 0$ (since, otherwise, $|5K|$
gives a birational map).

\item[(4)]  
By running
the minimal model program, one gets a relative minimal model $X\to
C$ of $a_V$ where $X$ has $\bQ$-factorial terminal singularities.
Then $K_{X/C}$ is nef (see, for instance, Ohno \cite[Theorem
1.4]{Ohno}), which means that $X$ is minimal since $K_C$ is trivial. In
the proof of Theorem \ref{6K}, we may and do replace $V$ by a
minimal model $X$ (i.e. $K_X$ nef) which has at worst
$\bQ$-factorial terminal singularities.
\end{itemize}

\begin{cor} Suppose $V$ (or $X$) satisfies \ref{reduc}(2) and \ref{reduc}(3). Then $q(X)=1$, $p_g(X)=h^2(\OO_X)\leq 2$ and thus $\chi(\OO_X)=0$.
\end{cor}
\begin{proof}
 Clearly one has $q(V)=1$. Since $q(S)=0$, we see $h^2(\OO_V)=h^1(a_*\omega_V)$. So one has
$\chi(\OO_V)=h^2(\OO_V)-p_g(V)=h^1(a_*\omega_V)-h^0(a_*\omega_V)=-\deg(a_*\omega_{V/C})\leq 0$ by the semi-positivity theorem of Fujita \cite{F}. Thus $\chi(\OO_V)=0$ and $p_g(V)=h^2(\OO_V)$. Also by the semi-positivity of $a_*\omega_{V}=a_*\omega_{V/C}$, $p_g(V)=h^2(\OO_V)=h^1(a_*\omega_V)\leq \text{rk}(a_*\omega_V)=2$. By Reid's R-R formula in \cite{YPG}, one can see $P_2(V)>0$ and $P_{m+1}(V)>P_{m}(V)$ for all $m\geq 2$.
\end{proof}

\subsection{Definitions and lemmas}\label{per}


Before proving the main result, we would like to recall
some notion and results in Chen and Hacon \cite{JC-H-mz}.

\begin{defn}
For any vector bundle $E$ on an elliptic curve, we write $E =
\oplus E_i$, where each $E_i$ is indecomposable. We define
$\nu(E):=\min \{  \mu(E_i) \}$, where
$\mu(E_i)=\frac{\deg(E_i)}{\text{rk}(E_i)}$ is the slope of $E_i$.
\end{defn}

\begin{defn} A coherent sheaf $\mathcal {F}$ on an abelian
variety $A$ is said to be $IT^0$ if $H^i(A,\mathcal
{F}\otimes P)=0$ for all $i>0$ and all $P\in
\text{Pic}^0(A)$.
\end{defn}

\begin{lem}\label{mucomp}(\cite[Lemma 4.8]{JC-H-mz})
Let $E_1$, $E_2$ be vector bundles on an
elliptic curve.
\begin{itemize}
\item[(1)] If $E_1$, $E_2$ are indecomposable and ${\rm
Hom}(E_1,E_2) \neq 0$, then $\mu(E_2) \ge \mu (E_1)$.

 \item[(2)] If there exists a surjective map
$E_1 \to E_2$, then $\nu(E_2) \ge \nu (E_1)$ .
\end{itemize}
\end{lem}

\begin{lem}\label{numin} (\cite[Lemma 4.10]{JC-H-mz})
Let $E$ be an $IT^0$ vector bundle on an elliptic curve which admits a short exact sequence $$ 0 \to F \to E \to Q \to 0$$ of coherent
sheaves such that $Q$ has generic rank $=0$ (resp. $\le 1$). Then
 $\nu(E) \ge \nu(F)$ (resp. $\nu(E)\geq \min\{1, \nu(F)\}$).
\end{lem}

\subsection{Multiplication maps $\varphi_{m,n}$ and
$\psi_{m,n}$.} Consider the fibration $a:X\to C$ as in \ref{per}. Let $F$ be a general fiber $F$ of $a$. Let $R_m:= H^0(F, \omega_F^m)$ and $E_m:=a_*
\omega_X^m$. By Chen and Hacon \cite[Lemma 4.1]{JC-H-mz}, $E_m$ is an $IT^0$ vector bundle of rank $P_m(F)$ for all $m\geq 2$. We also remark that $\nu(E_m) \ge 0$ by the semi-positivity theorem (see, for instance, Viehweg \cite{V1}) and Atiyah's description of
vector bundles over elliptic curves (cf. \cite{At}). We
consider the multiplication map of pluricanonical systems
on the fiber $F$, say
$$\varphi_{m,n}: R_m \ot R_n \to R_{m+n}.$$
This naturally induces a map between vector bundles
$$\psi_{m,n}: E_m \ot E_n \to E_{m+n}$$
where $m,\ n> 0$.
Clearly if cokernel of $\varphi_{m,n}$ has dimension $\leq
r$, then cokernel of $\psi_{m,n}$ has rank $\leq r$.

\subsection{Proof of Theorem \ref{6K}} {}First of all, we recall that the linear system
$|6K_V|$ separates two general points on two distinct general fibers of the Albanese map $a_V$ (see \cite[Theorem 2.8 (2)]{JC-H-mz}). Hence we just need to show that $|6K_V|$
separates two general points on a general fiber of $a_V$ to conclude the proof of Theorem \ref{6K}.

We now take the birational model $a: X\rightarrow C $ of $V$ as in \ref{reduc}(1)$\sim$(4).

\medskip



{\it Step 1}. We construct a relative canonical model $ W \to C$
of $a$.

We may take an integer $m \gg 0$ and pick a  very ample divisor $L$ on $C$ so that
\begin{itemize}
\item[i.] for the general fiber $F$ of $a$, $|mK_F|$ is base point free and $\Phi_{|mK_F|}(F)$ is the canonical model of $F$;

\item[ii.] $|a^*L+mK_X|$ is  free; \item[iii.]
$a_*\omega_X^m\otimes \OO_C(L)$ is generated by global sections
and then the restriction map $H^0(X, a^*L + mK_X) \to H^0(F,
mK_F)$ is surjective for general $F$;

\item[iv.] $a_*\omega_X^2\otimes \OO_C(L)$ is generated by global
sections  and then the restriction map $H^0(X, a^*L + 2K_X) \to
H^0(F, 2K_F)$ is surjective for general $F$.
\end{itemize}
The linear system $|a^*L+mK_X|$ defines a morphism $X \to \bP^N$
over $C$ and let $W$ be its image. Then we get a relative
canonical model $g: W\to C$. Clearly, by definition, $a$ factors
through $g$. Denote by $G$ the general fiber of $g$.  Then $W|_G$
is exactly the canonical model of $F$ for general $F$.
\medskip

{\it Step 2}. The relative bicanonical map $h:Y\to C$ of $g$.

It is known (cf. Catanese \cite[1.3 Example]{Cat}) that the
canonical model $G$ of any (1,2) surface  is a degree 10 weighted
hypersurface, with at worst rational double points, in $
\bP(1,1,2,5)$. Namely, if $x,y,z,u$ are coordinates of
$\bP(1,1,2,5)$, then $G$ is given by the homogeneous equation
$u^2-f_{10}(x,y,z)$ for some homogeneous polynomial
$f_{10}(x,y,z)$ of degree 10 in $x,y,z$. Furthermore the
bicanonical map $\varphi_2$ of $G$ is a double covering onto
$\bP(1,1,2)$ branched along a reduced divisor $B_0=
\text{div}(f_{10}) \subset \bP(1,1,2)$ of degree 10.

By the choice of $m$, we may assume that the rational map
$$\Phi_{|a^*(L)+2K_X|}:X\dashrightarrow Y$$ factors through $W$
where $Y$ is assumed to be the closure of the image.
Notice also that $a: X \to C$ factors through $Y$. Moreover, there
is a natural injection $Y \hookrightarrow \bP(a_*
\omega^2)=\bP(E_2)$, where $\bP(E_2)$ is a $\bP^3$-bundle over
$C$. We have a new fibration $h:Y\to C$ which is induced from the
bicanonical map of $g$.

Let $H$ be the general fiber of $h: Y \to C$. Over a general point of $C$, we have morphisms $F \to G \to H$ where $F$ is a minimal (1,2) surface, $G$ is the degree 10 hypersurface in $\bP(1,1,2,5)$ with RDPs and $H \cong \bP(1,1,2)$. We have seen that both $X \dashrightarrow Y$ and $W \dashrightarrow Y$ are well-defined
over general points of $C$. Replacing both $X$ and $W$ with suitable birational models $\hat{X}$ and $\hat{W}$ by a necessary birational modification to those
indeterminancies, we have the following commutative diagram:
$$\begin{CD}
\hat{X}  @>\sigma>> \hat{W} @>\tau>> Y  @>>> \bP(E_2)\\
@V\hat{a}VV  @V\hat{g}VV @VhVV  @VVpV\\
 C @>=>> C @>=>> C @>=>> C.
\end{CD}$$
where $\hat{X}$ (resp. $\hat{W}$) coincides with $X$ (resp. $W$)
over a Zariski open subset $U$ of $C$ and $\hat{a}$ (resp.
$\hat{g}$) factors through $a$ (resp. $g$).
\medskip

{\it Step 3}. The decomposition of $E_m$ by the double covering
construction.

Shrinking $U$, if necessary, so that $\tau: W_U=\hat{W}_U \to Y_U$ is a double covering branched along an even reduced divisor $B_U \subset Y_U$. Let $B_1$ be the closure of $B_U$ in $Y$. Then
$$\cO_{Y}(B_1)= \cO_{\bP(E_2)}(10) \otimes p^* \cM |_{Y}$$ for some line bundle $\cM$ on $C$. Set $B=B_1$ (resp. $B=B_1+H_0$) if
$\deg(\cM)$ is even (resp. odd), then $\cO_Y(B) = \cL^{\otimes 2}$, where $\cL=(\cO_{\bP(E_2)}(5) \otimes \pi^* \cM') |_Y$ for some $\cM'$.

Let $\mu: \tilde{Y} \to Y$ be the log resolution of $(Y,B)$ and let $\tilde{B}:= \mu^*B - 2 \lfloor \frac{\mu^*B}{2} \rfloor$ and $\tilde{\cL}=\mu^* \cL \otimes \cO(-\lfloor \frac{\mu^*B}{2} \rfloor )$. Clearly $\tilde{B}$ is a reduced {\it SNC} divisor and $\cO( \tilde{B}) = \tilde{\cL}^{\otimes 2}$.  Let  $\tilde{\pi}: \tilde{X} \to \tilde{Y}$ be the double cover over $\tilde{Y}$ branched along $\tilde{B}$. One sees that $\tilde{X}$ has at worst canonical singularities by local consideration.
We thus have
$$ \tilde{\pi}_* \cO_{\tilde{X}}(mK_{\tilde{X}}) = \cO_{\tilde{Y}} ( mK_{\tilde{Y}})\otimes \tilde{\cL}^{m} \oplus  \cO_{\tilde{Y}} ( mK_{\tilde{Y}}) \otimes \tilde{\cL}^{m-1}$$
for all $m>0$. Now if we take a common birational modification to both $\hat{X}$ and $\tilde{X}$ and take push-forwards in two directions respectively, we shall get
the following decomposition
$$ E_m:= E_{m,0} \oplus E_{m,1},$$
where
$$\begin{array}{l} E_m:=a_*\cO_X(mK_{X}); \\ E_{m,0}:= h_* \mu_* (\cO_{\tilde{Y}} ( mK_{\tilde{Y}})\otimes \tilde{\cL}^{m});\\
E_{m,1}:= h_* \mu_*( \cO_{\tilde{Y}} ( mK_{\tilde{Y}}) \otimes \tilde{\cL}^{m-1}).
\end{array}
$$
\medskip

{\it Step 4}. Calculating $\nu(E_{6,i})$.

It is rather easy to check that $\text{rk}(E_{m,0})=h^0(H, \cO(m))$ and
$\text{rk}(E_{m,1})= h^0(H, \cO(m-5))$ for a general fiber $H$ of $h$. Indeed, for $t \in U$,
\begin{eqnarray*}
E_m \otimes k(t) &\cong &H^0( F_t, \cO_{F_t}(mK_X)) \cong
H^0(G_t, \cO_{G_t}(mK_W))\\
& \cong& H^0( \bP(1,1,2,5), \cO(m)),
\end{eqnarray*}
$$E_{m,0} \otimes k(t) \cong H^0( H_t, \cO_{H_t}(mK_Y+mL)) \cong
H^0(\bP(1,1,2), \cO(m))\ \text{and}$$
$$E_{m,1} \otimes k(t) \cong H^0( H_t, \cO_{H_t}(mK_Y+(m-1)L)) \cong
H^0(\bP(1,1,2), \cO(m-5)).$$ It follows that $\psi_{m,n}$ induces
a map $$E_{m,0} \otimes E_{n,0} \to E_{m+n,0}.$$ Since $E_{m,0} =
E_{m}$ for $m \le 4$. One sees that
$$\psi_{4,2}: E_4 \otimes E_2 \cong E_{4,0} \otimes E_{2,0} \to
E_{6,0}$$ is generically surjective. Since $E_2$ is a non-zero
$IT^0$ sheaf, we have $h^0(E_2) \ge 1$. Hence $\nu(E_2) \ge
\frac{1}{4}$. Since $\psi_{2,2}$ is generically surjective, we
have $\nu(E_4)\geq \nu(E_{4,0}) \ge \frac{1}{2}$ by Lemma \ref{numin}. Similarly,
$\nu(E_{6,0}) \ge \frac{3}{4}$. Moreover, $E_{6,1}$ is $IT^0$ of
rank $2$ by Lemma \ref{mucomp}, hence $\nu(E_{6,1}) \ge \frac{1}{2}$.
\medskip

{\it Step 5}. Birationality of $\varphi_6$.

We need the following:

\begin{lem}\label{im} Let $\cF$ be a coherent sheaf on $X$ and $\cE:=a_*\cF$ on $C$. Suppose that $\cE$ is
an $IT^0$ vector bundle. Then for any general fiber $X_t$,  the image of the restriction map $H^0(C, \cE)\cong H^0(X,\cF)
\stackrel{res}{\to} H^0(X_t, \cF|_{X_t})$ has dimension $\ge
rk(\cE) \cdot \min\{\nu(\cE), 1 \}$.
\end{lem}
\begin{proof} Take the decomposition of
$\cE = \oplus \cE_i$ into indecomposable bundles. For
each $i$, there is an induced exact sequence
$$ 0 \to \cE_i \otimes \cO_C(-t) \to \cE_i \to \cE_i \otimes k(t) \to
0.$$
  Let $d_i =\deg( \cE_i)$ and $r_i=\text{rk} (\cE_i)$, then $\cE_i \otimes \cO_C(-t)$ has rank $r_i$ and degree $d_i-r_i$.
If $d_i=r_i$, then $\cE_i \otimes \cO_C(-t)$ is a indecomposable
rank $r_i$ vector bundle of degree $0$. Hence $\cE_i \otimes
\cO_C(-t) \cong U_{r_i} \otimes P$ for some $P \in \Pic^0(C)$ and
$U_{r_i}$ is a unipotent vector bundle (cf. \cite{At}). Whenever
$P= \cO$, we pick $t' \ne t$ and consider  $\cE_i \otimes
\cO_C(-t')\cong  U_{r_1} \otimes \cO(-t'+t)$ instead so that it
has no global section. Hence we may and do assume that $H^0(\cE_i
\otimes \cO_C(-t)) =0$ for general $t \in C$ if $d_i=r_i$.

 It  now follows that $h^0(\cE_i \otimes \cO_C(-t))= \max\{0, d_i-r_i\}$ for general $t$. Hence the image of $H^0(\cE_i) \to H^0(\cE_i \otimes k(t))$ has dimension $d_i$ (resp. $r_i$) if $d_i < r_i$ (resp. $d_i \ge r_i$). The statement now follows by simply taking the sum.
\end{proof}

Let $V_{m,i}$ ($i=0,1$) be the image of the following map
$$H^0( C, E_{m,i})) \hookrightarrow
H^0(C,E_m) \stackrel{res}{\to} H^0(F_t, \cO(mK)|_{F_t})$$
for a general point $t\in C$.
Then we have $ \dim V_{6,0} \ge 12$ and $\dim V_{6,1} \ge 1$ by
Lemma \ref{im}.
\medskip

{\bf Claim}. {\em The subsystem given by the vector space
$$V_{6,0}+V_{6,1} \subset H^0(G_t, \cO(6))$$ gives a birational
map on $G_t$ for all general $t\in C$.}
\medskip

We consider the local sections explicitly. Let $x,y,z, u$ be all
the 4 coordinates of $\bP(1,1,2,5)$ with weights $1,1,2,5$. Then
$E_{m,0} \otimes k(t)$ is generated by sections in $\{x^iy^jz^k|
i+j+2k=m\}$ and $E_{m,1} \otimes k(t)$ is generated by sections in
$\{x^iy^jz^ku| i+j+2k=m-5\}$. In a word, either $xu$ or $yu$
extends to global sections in $H^0(X,6K_X)$. Furthermore, at least
12 linearly independent sections in $E_{m,0} \otimes k(t)$ can be extended
to global sections in $H^0(X,6K_X)$.

 To prove the claim, we put $H=H_t$ and let $\Sigma_0 \subset H^0(H, \cO_H(6))$
(resp. $\Sigma_1\subset H^0(H, \cO_H(6))$) be the subspace spanned
by $\{x^6,\cdots,y^6\}$ (resp. by $\{ x^4z,x^3yz,\cdots,y^4z\}$).
We see that $\dim\Sigma_0=7$ and $\dim\Sigma_1=5$. By dimensional
considerations, one has $\dim V_{6,0} \cap \Sigma_0 \ge 3$ and
$\dim V_{6,0} \cap \Sigma_1 \ge 1$. Pick linearly independent
elements $\sigma_{0,1},  \sigma_{0,2}, \sigma_{0,3} \in
V_{6,0} \cap \Sigma_0$ and $ z \sigma_{1} \in V_{6,0} \cap
\Sigma_1$. We consider the map $\tilde{\varphi}: H \dashrightarrow
\bP^3$ defined by these $4$ sections. It is easy to see that
$\tilde{\varphi}$ has image of dimension $2$. Indeed, consider the
map $\varphi: H \dashrightarrow \bP^{11}$ given by $V_{6,0}$ with
image $H'$. Since $\varphi$ factors through $\tilde{\varphi}$, one
sees that $H'$ is a surface and clearly $\deg(H') \ge 10$. Since
$$\deg(\varphi) \cdot \deg(H')\leq (\cO_H(6)\cdot \cO_H(6))_H=18,$$ it follows that $\varphi$ has degree $1$, hence is
birational.

Since $G_t \cong X_{10} \to \bP(1,1,2) \cong H$ is a $2:1$ map and
$u$ can separate points on general fibers of this double covering.
Hence the sections in $V_{6,1}$ separate points on general
fibers of this double covering.

The Claim now follows and hence  this completes the proof of
Theorem \ref{6K}.\qed

\begin{exmp} Suppose that there exists a minimal irregular threefold $X$ with a fibration $f:X \to C$ fibered by $(1,2)$ surfaces.
Suppose that $K_X^3=\frac{1}{2}$ and $B(X)=\{ 3 \times (1,2)\}$. By Reid's R-R formula, one has $P_2(X)=1$, $P_3(X)=2$, $P_4(X)=5$, $P_5(X)=9$ and $ P_6(X)=16$. We show that $|5K_X|$ may be non-birational.

Note that $\text{rk}(E_{5,0})=12$ and $\text{rk}(E_{5,1})=1$. Assume $h^0(E_{5,0})=8$ and $h^0(E_{5,1})=1$.

Now $ H^0(F_t, 5K|_{F_t})$ is generated by $$ \{ x^5,...,y^5,x^3z,x^2yz,xy^2z,y^3z,
xz^2, yz^2, u\}.$$  If $V_{5,1}$ is generated by $\{x^5,...,y^5, xz^2, yz^2\}$ and $V_{5,2}$ is generated by $u$, then these sections can not
distinguish points $(x_0,y_0,z_0, u_0)$ from $(x_0,y_0,-z_0, u_0)$. In other
words, it may only give a $2:1$ map on $F_t$ instead of
a birational map.

However, we do not know whether this kind of examples really exists or not.
\end{exmp}


\end{document}